\documentclass[12pt,oneside,reqno]{amsart}
\usepackage{graphicx}
\usepackage{amsfonts,amssymb,amscd,amsmath,enumerate,verbatim,calc}
\usepackage[T1]{fontenc}
 \usepackage[utf8]{inputenc}
\usepackage{float}

\newtheorem{Theorem}{Theorem}[section]
\newtheorem{Lemma}[Theorem]{Lemma}
\newtheorem{Corollary}[Theorem]{Corollary}
\newtheorem{Proposition}[Theorem]{Proposition}

\newtheorem{Example}[Theorem]{Example}

\newtheorem{Definition}[Theorem]{Definition}

\begin{document}
\title{The Fractional fixing number of graphs}
\author{Hira Benish$^1$, Iqra Irshad$^1$, Min Feng$^2$, Imran Javaid$^{1, *}$}
%\subjclass{Primary: , Secondary: }
\keywords{Fixing set, Fixing neighborhood, Fixing function, Fractional fixing number.\\
\indent 2010 {\it Mathematics Subject Classification.} 05C25\\
\indent $^*$ Corresponding author: imran.javaid@bzu.edu.pk}
%\indent 2012 {\it Mathematics Subject Classification.} \\
\address{$^1$Centre for advanced studies in Pure and Applied Mathematics,
Bahauddin Zakariya University, Multan, Pakistan.  Email:
hira\_benish@yahoo.com, iqrairshad9344@gmail.com,
imran.javaid@bzu.edu.pk.}
\address {$^2$ Sch. Math. Sci. Lab. Math. Comp. Sys. Beijing Normal
University, Beijing, 100875, China. Email: fgmn\_1998@163.com}

\maketitle
\begin{abstract}
An automorphism group of a graph $G$ is the set of all permutations
of the vertex set of $G$ that preserve adjacency and non-adjacency
of vertices in $G$. A fixing set of a graph $G$ is a subset of
vertices of $G$ such that only the trivial automorphism fixes every
vertex in $S$. Minimum cardinality of a fixing set of $G$ is called
the fixing number of $G$. In this article, we define a fractional
version of the fixing number of a graph. We formulate the problem of
finding the fixing number of a graph as an integer programming
problem. It is shown that a relaxation of this problem leads to a
linear programming problem and hence to a fractional version of the
fixing number of a graph. We also characterize the graphs $G$ with
the fractional fixing number $\frac{|V(G)|}{2}$ and the fractional
fixing number of some families of graphs is also obtained.
\end{abstract}
\section{Motivation and Background}
Motivation behind the development of the fractional idea has
multiple aspects. One of the interesting aspect is that the
fractional version multiplies the range of applications in operation
research, scheduling or in various kind of assignment problems.
Theorems in their fractional version are mostly easier to prove.
Mostly for fractional and classical coefficients of graphs bounds
are same or it may form conjecture in fractional version. Most of
the times, the conjecture becomes refined theorem in their
fractional version. Fractional version of parameters have drawn the
attention of researchers to a wealth of new problems and
conjectures. Fractional graph theory has modified the concept of
integer-valued graph theory to the non-integral values.\\

Interesting aspects of fractional graph theory motivated Hedetniemi
{\it et al.} to introduce the concept of the fractional domination
number of a graph by linear relaxation of the integer programming
problem of domination number of graphs \cite{HHW}. A variety of work
has been done on the fractional domination number of graphs, see
\cite{DCF, RW, RW2, Walsh}. Currie {\it et al.} defined the
fractional metric dimension of a graph as the optimal solution of
the linear relaxation of the integer programming problem of the
metric dimension of graphs \cite{CO}. The fractional metric
dimension of graphs and graph products has also been studied
\cite{AM, FLW, FW, FW2, KS, EY}. The metric dimension of a graph is
an upper bound for the fixing number of graph \cite{erw}. Their
relationship has been studied in \cite{cac,erw}. In this paper, we
introduce fractional version of the fixing number of graph by
introducing the idea of fixed graph. In the following paragraph, we
introduce some relevant terminology needed for exposition of
this idea.\\

All graphs considered in this paper are simple, non-trivial and
undirected. Let $G=(V(G),E(G))$, when there is no ambiguity, we
simply write $G=(V,E)$. The number of vertices and edges of $G$ are
called the {\it order} and the {\it size} of $G$ respectively. For
$u,v\in V(G)$, $u\sim v$ means $u$ and $v$ are adjacent and
$u\not\sim v$ means $u$ and $v$ are not adjacent. The {\it open
neighborhood} of a vertex $u$ is $N_G(u)=\{v \in V(G):$ $v\sim u$ in
$G\}$ and the {\it closed neighborhood} of $u$ is $N_G [u] = N_G(u)
\cup\{u\}$. The number $|N_G(v)|$ is called the degree of $v$ in
$G$. The {\it distance} $d(u,v)$ between two vertices $u,v \in V(G)$
is the length of a shortest path between them. Two distinct vertices
$u$ and $v$ in a graph $G$ are said to be {\it twins} if
$d(u,w)=d(v,w)$ for all $w\in V(G)\backslash \{u,v\}$. A set
$U\subseteq V(G)$ is called a {\it twin-set} of $G$ if $u, v$ are
twins in $G$ for every pair of distinct vertices
$u,v\in U$.\\

A {\it permutation} of a set is a bijection from the set to itself.
An {\it automorphism} of a graph is a permutation of the vertex set
that preserves adjacency and non-adjacency of the vertices. An
equivalent definition is: $\pi:V(G)\rightarrow V(G)$ is an
automorphism of a graph $G$ if for all $u,v\in V(G)$, $\pi(u)\sim
\pi(v)$ if and only if $u\sim v$. The set of all automorphisms of
$G$ forms a group, called the {\it automorphism group} of the graph
$G$. We use $\Gamma(G)$, or $\Gamma$ if $G$ is clear from the
context, to denote the full automorphism group of a graph $G$. We
consider the full automorphism group $\Gamma$ acting on the vertex
set $V$ of $G$. For $u\in V$, the {\em orbit} $O(u)$ and  {\em
stabilizer} $\Gamma_u$ of $u$ is defined as
$O(u)=\{\pi(u):\pi\in\Gamma\}$ and
$\Gamma_u=\{\pi\in\Gamma:\pi(u)=u\}$. For $T\subseteq V$,
$\Gamma_T=\cap_{u\in T}\Gamma_u$. For $x\in V$, the subgroup
$\Gamma_x$ has a natural action on $V$ and the orbit of $u$ under
this action is denoted by $O_x(u)$ {\it i.e.},
$O_x(u)=\{\pi(u):\pi\in\Gamma_x\}$. Define
$$A(G)=\{u:|O(u)|\ge2\} \qquad\textup{and}\qquad C(G)=\{u:|O(u)|=1\}.$$
Then $V(G)$ is the disjoint union of $A(G)$ and $C(G)$. Define
$$V_a(G)=\{(u,v)\in A(G)\times A(G):O(u)=O(v),u\ne v\}.$$ If $G$ is a
rigid graph (a graph with $\Gamma(G)={id}$), then
$V_a(G)=\emptyset$.\\

A set $\mathcal{D}\subseteq V(G)$ is called a {\it determining set}
of $G$ if whenever $\alpha,\beta\in \Gamma(G)$ such that
$\alpha(v)=\beta(v)$ for all $v\in \mathcal{D}$, then
$\alpha(u)=\beta(u)$ for every $u\in V(G)$. The {\it determining
number} of a graph $G$ is the order of a smallest determining set,
denoted by $Det(G)$. Determining sets of graphs were introduced by
Boutin in \cite{bou}. She gave several ways of finding and verifying
determining sets. The natural lower bounds on the determining number
of some graphs were also given. Determining sets are frequently used
to identify the automorphism group of a graph. For further work on
determining sets and its relation with other parameters, see
\cite{bou,cac}. Erwin and Harary independently introduced an
equivalent concept: the fixing number of a graph $G$ \cite{erw}. A
set $S\subset V$ is a {\it fixing set} of $G$ if $\Gamma_S$ is
trivial, i.e., the only automorphism that fixes all vertices of $S$
is the trivial automorphism. The cardinality of a smallest fixing
set is called the {\it fixing number} of $G$, denoted by $fix(G)$.
The fixing number of graphs is also used to study the symmetry of
graphs and the relationships of groups and graphs \cite{gib}. The
equivalence of determining and fixing set of graphs was also
established in \cite{gib}.

A fixing set $S$ of $G$ is minimal if no proper subset of $S$ is a
fixing set of $G$. In families of graphs like path, cycle, complete
graph and complete bipartite graphs minimum fixing sets and minimal
fixing sets have same cardinality. This is not the case always. For
example, let $C_{4n}$, $n\geq2$, be the cycle graph with
$V(C_{4n})=\{v_1,v_2,\cdots,v_{4n}\}$. Attach two pendent vertices
with $v_1$ and $v_{2n+1}$. Let us denote the resulting graph by
$\mathcal{C}$. It is easy to check that $\{v_2\}$ and
$\{v_{n+1},v_{2n+1}\}$ are minimal fixing sets of $\mathcal{C}$. In
this example, there exists a minimal fixing set whose cardinality is
different from that of a minimum fixing set. We call the maximum
cardinality of a minimal fixing set of a graph $G$, the upper fixing
number of $G$, denoted by $fix^+(G)$. Note that $fix(G)\leq fix^+(G)$.\\

This paper is organized as follows: in section 2, we define the
fixing neighborhood of a graph and the fixed graph. In this section
fractional version of the fixing number of a graph is defined. We
give an integer programming problem for the problem of finding the
fixing number of a graph. We also show that by relaxing conditions
of this problem, a linear programming problem is formulated which is
a fractional version of the fixing number of a graph. In section 3,
we characterize the graphs $G$ with the fractional fixing number
$\frac{|V(G)|}{2}$. Section 4 of this paper is devoted to the study
of the fractional fixing number of some families of graphs. In
section 5, we study the fractional fixing number of corona product
of graphs. In the last section, we study the fractional fixing
number of composition product of graphs.
\section{Fixing Neighborhood and Fixed Graphs}
In this section, we define the concept of fixing neighborhood and
fixed neighborhood. We also define the fixed graph using fixing
neighborhood of graph.

A vertex $v$ is {\it fixed} by an automorphism $\pi\in \Gamma(G)$ if
$\pi\in \Gamma_v$. A vertex $x$ is said to be fixed vertex in $G$ if
$\pi(x)=x, \,\,\forall \, \pi\in \Gamma(G)$. A vertex $x\in V(G)$ is
said to {\it fix a pair} $(u,v)\in V(G)\times V(G)$, if $O_x(u)\ne
O_x(v)$ in $G$. For $(u,v)\in V(G)\times V(G)$ and the set
$F(u,v)=\{x\in V(G): O_x(u)\ne O_x(v)\}$ is called the {\it fixing
neighborhood} of $(u,v)$. For each $x\in V(G)$, the set
$F(x)=\{(u,v)\in V(G)\times V(G) :O_x(u)\ne O_x(v)\}$ is called the
{\it fixed neighborhood} of $x$. For any two distinct vertices $u$
and $v$ in $G$ with $O(u)\neq O(v)$,  $F(u,v)=V(G)$.
 Note that $fix(G)=0$ if and only if $G$
is rigid graph \cite{gib}. If $G$ is rigid graph then for all
distinct vertices $u,v \in V(G)$, $F(u,v)=V(G)$ but converse is not
true. For example, consider an even path $P_n$ on $n$ vertices, then
$F(u,v)=V(G)$ for any two distinct
vertices of $P_n$ which is not rigid.\\
For two distinct vertices $u$ and $v$ in a graph $G$, define
$\iota_{u,v}: V(G)\rightarrow V(G),$
$$\iota_{u,v}(x)=\left\{\begin{array}{ll}
                                                    v, & \textup{if } x=u, \\
                                                    u, &  \textup{if } x=v, \\
                                                    x, & \textup{otherwise.}
                                                  \end{array}\right.
$$
Then $\iota_{u,v}$ is an automorphism of $G$ if and only if $u$ and
$v$ are twins. Hence, we have the following result;
\begin{Lemma}\label{twins}
Let $u$ and $v$ be distinct vertices in a graph $G$. Then
$\{u,v\}\subseteq F(u,v)$. Moreover, we have $F(u,v)=\{u,v\}$ if and
only if $u$ and $v$ are twins.
\end{Lemma}
For two distinct vertices $u$ and $v$ in $G$, $R(u,v)=\{x\in V(G):
d(x,u)\ne d(x,v)\}$ where $d(x,u)$ is the distance between $x$ and
$u$.
\begin{Lemma}\label{lemma}
  Let $u$ and $v$ be two distinct vertices in $G$. Then $R(u,v)\subseteq F(u,v)$.
\end{Lemma}
\proof If $x\in R(u,v)$, then $d(x,u)\ne d(x,v)$, which implies that
$O_x(u)\ne O_x(v)$ since   automorphisms preserve the distances in
the graph $G$. Hence, the required result follows. \qed

Note that in order to destroy automorphisms, only those vertices
$u,v\in V(G)$ are of interest for which $O(u)=O(v)$ and $|O(u)|\geq
2$. So it is sufficient to consider $V_a(G)$ instead of $V(G)\times
V(G)$. If $S$ is a fixing set, then it is clear that $S\cap
F(u,v)\neq \emptyset$ for any pair $(u,v)\in V_a(G)$.
%Hence, our terms $A(G)$, $C(G)$, and $V_a(G)$ are justified.
Moreover, for each pair in $V_a(G)$ can be fixed by elements of
$A(G)$ only. If $F(u,v)=A(G)$, for all $u,v \in A(G)$, then
$fix(G)=1$ but converse is not true. To see this, consider the
cartesian product of $P_4$ and $P_5$, denoted by $P_4\Box P_5$ and
$V(P_4\Box P_5)=\{u_{ij}|1\leq i\leq 4, 1\leq j\leq 5\}$. Note that
$A(P_4\Box P_5)=V(P_4\Box P_5)$ and $fix(P_4\Box P_5)=1$. But
$F(u_{11},u_{15})\neq A(P_4\Box P_5)$ because $u_{13}$ does not fix
the pair $(u_{11},u_{15})$.

 The {\it fixed graph}, $I(G)$, of a
graph $G$ is a bipartite graph with bipartition $(V(G),V_a(G))$ and
a vertex $x\in V(G)$ is adjacent to a pair $(u,v)\in V_a(G)$ if
$x\in F(u,v)$. For a set $D\subseteq A(G)$, $N_{I(G)}(D)=\{(u,v)\in
V_a(G):$ $x\in F(u,v)$ for some $x\in D\}$. In the fixed graph,
$I(G)$, the minimum cardinality of a subset $D$ of $V(G)$ such that
$N_{I(G)}(D)=V_a(G)$ is the fixing number of $G$. For a graph $G$ of
order $n$, if $C(G)=V(G)\setminus A(G)$ and $F(u,v)=A(G)$ for all
$u,v\in A(G)$, then $I(G)=K_{|A(G)|, |V_a(G)|}\cup
\overline{K_{|C(G)|}}$. For a path $P_n$ on even $n$ vertices,
$I(G)=K_{n,\frac{n}{2}}$ and for a path $P_n$ on odd $n$ vertices,
$I(G)=K_{n-1,{\frac{n-1}{2}}}\cup K_1$. Least positive integer $k$
such that every $k$-set of vertices of a graph $G$ is a fixing set
of $G$ is called the {\it fixed number} of $G$ denoted by $fxd(G)$.
A graph $G$ is said to be a $k$-fixed graph if $fix(G) = fxd(G) =
k$. Javaid {\it et al.} studied fixed number of graphs in
\cite{ijav}. Lower and upper bounds on the cardinality of edge set
of $I(G)$ for a $k$-fixed graph $G$ were given in \cite{ijavh} and
\cite{ijav}.
\begin{Proposition}
If $G$ is a $k$-fixed graph of order $n\geq 2$, fixing number $k$
and $|A(G)|=l$, then
$$\frac{l}{2}(l-k+1)\leq |E(I(G))|\leq n({n \choose 2}-k+1).$$
\end{Proposition}
Rest of this section is devoted to the formulation of fractional
version of the fixing number of a graph and its integer programming
formulation.\\
Suppose $V(G)=\{v_1,v_2,...,v_n\}$ and
$V_a(G)=\{s_1,s_2,...,s_{r}\}$, $r\geq 1$. Let $B=(b_{ij})$ be the
${r}\times n$
matrix with\\

$$b_{ij}=\left\{
  \begin{array}{ll}
             1,   \,\,\,\     & \mbox{if}\,\,\,\ s_iv_j\in E(I(G)),\\
             0,    \,\,\,\     & \mbox{otherwise},\\
            \end{array}
             \right.
 $$
for $1\leq i\leq r$ and $1\leq j \leq n$.\\
The integer programming formulation of the fixing number is given
by:
\begin{center}
Minimize $f(x_1,x_2,...,x_n)=x_1+x_2+...+x_n,$\\
subject to the constraints\\
$Bx\geq [1]_{r}\qquad\textup{ and }\qquad x_i\in \{0,1\}$
\end{center}
where $x=[x_1,x_2,...,x_n]^T$, $[1]_k$ is the $k\times 1$ matrix all
of whose entries are 1, and $[0]_n$ is the $n\times1$ matrix all of
whose entries are 0.\\
If we relax the condition, $x_i\in \{0,1\}$ for every $i$ and
require that $x_i\geq 0$ for all $i$, then we obtain the following
linear programming problem:
\begin{center}
Minimize $f(x_1,x_2,...,x_n)=x_1+x_2+...+x_n$\\
subject to the constraints\\
$Bx\geq [1]_{r}\qquad\textup{ and }\qquad x\geq[0]_n.$
\end{center}
In terms of the fixed graph $I(G)$ of $G$, solving this linear
programming problem amounts to assigning non-negative weights to the
vertices in $V(G)$ so that for each pair in $V_a(G)$, the sum of
weights in its neighborhood is at least 1 and such that the sum of
weights of the vertices of $G$ is as small as possible. The smallest
value for $f$ is called the {\it fractional fixing number} of $G$.
\begin{Definition} Let $G$ be a connected graph of order $n$. A
function $g:V(G)\rightarrow [0,1]$ is a {\it fixing function} $FF$
of $G$ if $g(F(u,v))\geq1$ for any pair $(u,v)\in V_a(G)$, where
$g(F(u,v))=\sum\limits_{x\in F(u,v)}g(x)$ and $|g|=\sum\limits_{v\in
V}g(v)$. The {\it fractional fixing number}, denoted by $fix_f(G)$,
is the minimum value of $FF$.
\end{Definition}

\begin{Definition}
Let $G$ be a connected graph of order $n$. A function
$g:V(G)\rightarrow [0,1]$ is a {\it resolving function} $RF$ of $G$
if $g(R(u,v))\geq1$ for two distinct vertices $u,v\in V(G)$, where
$g(R(u,v))=\sum\limits_{x\in R(u,v)}g(x)$ and $|g|=\sum\limits_{v\in
V}g(v)$. The {\it fractional metric dimension}, denoted by
$dim_f(G)$, is the minimum value of $RF$.
\end{Definition}
%given by $dim_f(G)=\min\{|g|:g \mbox{is\,\, a\,\,resolving\,\, function\,\, of\,\, G}\}$.
In the next theorem, we
show that $dim_f(G)$ is an upper bound of $fix_f(G)$.
\begin{Theorem}
  For any connected graph $G$, we have $fix_f(G)\le dim_f(G)$.
\end{Theorem}
\proof If $G$ is a rigid graph, the $fix_f(G)=0\le dim_f(G)$. Now
suppose that $G$ is a non-rigid graph. By Lemma~\ref{lemma}, each
resolving function of $G$ is a fixing function of $G$. Therefore,
our desired inequality holds. \qed

\section{Characterization of graphs with $fix_f(G)=\frac{|V(G)|}{2}$}
In this section, we characterize the graphs having
$fix_f(G)=\frac{|V(G)|}{2}$. For graphs with $fix(G)=1$, it follows
that $fix_f(G)=1$ because the characteristic function of a minimal
fixing set is an $FF$ of $G$, it follows that $1\leq fix_f(G)\leq
fix(G)\leq fix^+(G)\leq n-1$. Note that fixing function plays an
important role while finding fractional fixing number of a graph. To
define fixing function that meets all conditions, we need to know
cardinalities of fixing neighborhoods of $(u,v)\in V_a(G)$. For a
graph $G$ of order $n$, we define
\begin{center}
$f(G)=\min\{|F(u,v)|: (u,v)\in V_a(G)\}.$
\end{center}
Now, we express the fractional fixing number of $G$ in terms of
$f(G)$ in the following proposition:
\begin{Proposition}\label{l1}
Let $G$ be a connected graph of order $n$. Then
$fix_f(G)\leq\frac{n}{f(G)}$.
\end{Proposition}
\begin{proof}
Let $g:V(G)\rightarrow[0, 1]$ defined by $g(v)=\frac{1}{f(G)}$. For
any two distinct vertices $u$ and $v$, we have
$g(F(u,v))=\frac{|F(u,v)|}{f(G)}\geq1$. Clearly $g$ is a fixing
function of $G$. Hence, $fix_f(G)\leq |g|=\frac{n}{f(G)}$.
\end{proof}
By above proposition and Lemma \ref{twins}, we have the following
result:
\begin{Corollary}\label{cro1n}
For a connected graph $G$ of order $n$, we have $fix_f(G)\leq
\frac{n}{2}$.
\end{Corollary}
In the rest of this section, we characterize all graphs $G$
attaining the upper bound in Corollary \ref{cro1n}.
\begin{Lemma}\label{cg}
For a non-rigid graph $G$, we have
$$
fix_f(G)\le \frac{|V(G)|-|C(G)|}{2},
$$
\end{Lemma}
\begin{proof}
Define a function $g:V(G)\rightarrow[0,1]$,
$$g(x)=\left\{\begin{array}{ll}
                                            0, & \textup{if }x\in C(G), \\
                                            \frac{1}{2}, &  \textup{otherwise.}
                                          \end{array}\right.
$$
Note that $|V(G)|-|C(G)|\ge 2$. Pick any two distinct vertices $u$
and $v$ in $G$. If $O(u)\ne O(v)$, then $F(u,v)=V(G)$, and so
$g(F(u,v))=|g|\ge 1$. If $O(u)=O(v)$, then $g(u)=g(v)=\frac{1}{2}$,
which implies that $g(F(u,v))\ge 1$ by Lemma~\ref{twins}. It follows
that $g$ is a fixing function. Hence, the desired result holds.
\end{proof}

Given a graph $H$ and a family of graphs $\mathcal{I}=\{I_v\}_{v\in
V(H)}$, indexed by $V(H)$, their {\em generalized lexicographic
product}, denoted by $H[\mathcal{I}]$, is defined as the graph with
the vertex set $ V(H[\mathcal{I}])=\{(v,w)|v\in V(H)\textup{ and }
w\in V(I_v)\} $ and the edge set $ E(H[\mathcal{I}])
=\{\{(v_1,w_1),(v_2,w_2)\}|\{v_1, v_2\}\in E(H), \textup{or }v_1=
v_2\textup{ and }\{w_1,w_2\}\in E(I_{v_1})\}.$

\begin{Theorem}\label{mthm}
Let $G$ be a non-trivial graph of order $n$. Then the following
conditions are pairwise equivalent.

  {\rm(i)} $fix_f(G)=\frac{n}{2}$.

  {\rm(ii)} Each vertex in $G$ has a twin.

  {\rm(iii)} There exist a graph $H$ and a family of graphs $\mathcal{I}=\{I_v\}_{v\in
V(H)}$, where $I_v$ is a non-trivial null graph  or a non-trivial
complete graph, such that $G$ is isomorphic to $H[\mathcal I]$.
\end{Theorem}
\begin{proof}
We show that (i) indicates (ii),
 (ii) indicates (iii), and (iii) indicates (i).
Suppose (i) holds. Then $C(G)=\emptyset$ by Lemma \ref{cg}. If there
exists a vertex $u$ in $G$ such that $u$ does not have a twin, then
the following function $g: V\rightarrow[0,1]$,
$$g(x)=\left\{\begin{array}{ll}
                                       0, & \textup{if }x=u, \\
                                       \frac{1}{2}, & \textup{if }x\ne u,
                                     \end{array}\right.
$$
is a fixing function of $G$ by Lemma \ref{twins}, which implies that
$fix_f(G)\le\frac{n-1}{2}$, a contradiction. So (ii) holds.

Suppose (ii) holds. For   $x,y\in V(G)$, define $u\equiv v$ if and
only if $x=y$ or $x, y$ are twins. It is clear that $\equiv$ is an
equivalence relation. Suppose
\begin{equation*}
O_{1}, \ldots, O_{m}
\end{equation*}
are the equivalence classes. Then the induced subgraph on each
$O_i$, denoted also by $I_{O_i}$, is a non-trivial null graph  or a
non-trivial complete graph. Let $H$ be the graph with the vertex set
$\{O_1,\ldots,O_m\}$, where two distinct vertices $O_i$ and $O_j$
are adjacent if  there exist $x\in O_i$ and $y\in O_j$ such that $x
$ and $y$ are adjacent in $G$. It is routine to verify that $G$ is
isomorphic to $H[\mathcal I]$, where $\mathcal
I=\{I_{O_i}:i=1,\ldots,m\}$. So (iii) holds.

Suppose (iii) holds. For $v\in V(H)$, write
$$
V(I_v)=\{w_v^1,\ldots,w_v^{s(v)}\}.
$$
Then $s(v)\ge 2$, and $(v,w_v^i)$ and $(v,w_v^j)$ are twins in
$H[\mathcal I]$, where $1\le i<j\le s(v)$. Let $h$ be a fixing
function of $H[\mathcal I]$ with $|h|=fix_f(H[\mathcal I])$. By
Lemma \ref{twins}, we get
\begin{equation*}
  h((v,w_v^i))+h((v,w_v^j))\ge 1\quad\textup{for }1\le i<j\le s(v),
\end{equation*}
which implies that
$$
\sum_{k=1}^{s(v)}h(v,w_v^k)\ge\frac{s(v)}{2},
$$
and so
$$
fix_f(G)=fix_f(H[\mathcal I])=|h|=\sum_{v\in
V(H)}\sum_{k=1}^{s(v)}h((v,w_v^k)) \ge \sum_{v\in
V(H)}\frac{s(v)}{2} =\frac{|V(H[\mathcal I])|}{2} =\frac{n}{2}.
$$
So (i) holds. We accomplish the proof.
\end{proof}
For following families of graphs, each vertex in these graph has a
twin. Using Theorem \ref{mthm}, we get the fractional fixing number
of these families of graphs:
\begin{Example}\label{c1}
$fix_f(G)=\frac{|V(G)|}{2}$ for each of the following graphs:
\begin{enumerate}
  \item $G=K_n,n\geq2$.
  \item $G=K_n-e$, where $n\ge 4$ and $e$ is an arbitrary edge of $K_n.$
  \item $G=K_{2t}-M,t\geq2$ and $M$ is a perfect matching in
$K_{2t}$.
  \item $G$ is complete $k$-partite graph $K_{n_1,n_2,...,n_k}$,
where $k\geq2$ and $n_i\geq 2$.
\end{enumerate}
\end{Example}
The \emph{join graph} $G+H$ is the graph obtained from $G$ and $H$
by joining each vertex of $G$ with every vertex of $H$. Note that,
if each vertex in $G_i$ has a twin for $i\in\{1,2\}$, then each
vertex in $G_1+G_2$ has a twin. Hence, we have:
\begin{Corollary}\label{c2}
Let $\Theta$ denote the collection of all connected graphs $G$ with
$fix_f(G)=\frac{|V(G)|}{2}$. If $G_1,G_2\in \Theta$, then
$G_1+G_2\in \Theta$.
\end{Corollary}
\begin{Corollary}\label{c3}
If $fix_f(G)=\frac{n}{2}$, then
$fix_f(G+\overline{K_k})=\frac{n+k}{2}$, for $k\geq 2$.
\end{Corollary}

\begin{Theorem}\label{t3}
Any connected graph $H$ can be embedded as an induced subgraph of a
connected graph $G$ with $fix_f(G)=\frac{|V(G)|}{2}$.
\end{Theorem}
\begin{proof}
Let $V(H)=\{u_1, u_2, \cdots, u_n\}$. Consider the graph $G$ formed
from $H$ by replacing each vertex $u_i$ of $H$ by $u_{i_1}$ and
$u_{i_2}$, and joining $u_{i_s}$ to $u_{j_t}$, where
$s,t\in\{1,2\}$, whenever $u_i$ and $u_j$ are adjacent in $H$.
Hence, $u_{i_1}$ and $u_{i_2}$ are twins in $G$, and so
$fix_f(G)=\frac{|G|}{2}$, and $H$ is an induced subgraph of $G$.
\end{proof}
\section{Fractional fixing number of some families of graphs}
In this section, we determine the fractional fixing number of some
families of graphs. A graph $G$ is \emph{vertex-transitive} if its
automorphism group $\Gamma(G)$ acts transitively on the vertex set.
For any two vertices $v$ and $w$ in $V(G)$, $\Gamma_v$ and
$\Gamma_w$ are isomorphic and index of $\Gamma_v$ in $\Gamma(G)$ is
equal to the order of $V(G)$. In the following result, we give the
fractional fixing number of a vertex-transitive graph $G$ in terms
of the parameter $f(G)$.
\begin{Theorem}\label{vertT} Let $G$ be a vertex-transitive graph, then
$fix_f(G)=\frac{|V(G)|}{f(G)}$.
\end{Theorem}
\begin{proof}
Let $f(G)=p$. Then there exists a pair of distinct vertices
$(u,v)\in V_a(G)$ such that $|F(u, v)|=p$. Suppose
$F(u,v)=\{r_1,r_2...,r_p\}$. For any automorphism $\alpha$ of $G$,
$F(\alpha(u),\alpha(v))=\{\alpha(r_1),\alpha(r_2),...,\alpha(r_p)\}$.
Let $h$ be a fixing function of $G$ with $fix_f(G)=|h|$. Then
$$h(\alpha(r_1))+h(\alpha(r_2))+...+h(\alpha(r_p))=h(F(\alpha(u),\alpha(v)))\geq
1,$$ which implies that
$$\sum\limits_{\alpha\in
\Gamma(G)}(h(\alpha(r_1))+h(\alpha(r_2))+...+h(\alpha(r_p)))\geq
|\Gamma(G)|.$$ Since $G$ is vertex transitive, we have
$$|\Gamma_{r_1}|.|h|+|\Gamma_{r_2}|.|h|+...+ |\Gamma_{r_p}|.|h |\geq
|\Gamma(G)|$$ which implies that $fix_f(G)\geq \frac{|V(G)|}{p}$. By
Proposition \ref{l1}, we have the required result.
\end{proof}
Since cycle $C_n$ of order $n$ is vertex transitive, therefore we
have the following result:
\begin{Corollary}
For the cycle $C_n$, we have $$fix_f(C_n)=\left\{\begin{array}{ll}
                                       \frac{n}{n-2}, & \textup{if $n$ is even }, \\
                                       \frac{n}{n-1}, & \textup{if $n$ is odd }.
                                     \end{array}\right.
$$
\end{Corollary}
A non-trivial connected graph $G$ is {\em distance-transitive} if
given any two ordered pairs of vertices $(u_1,v_1)$ and $(u_2,v_2)$
such that $d(u_1,v_1)=d(u_2,v_2)$, there is an automorphism $\sigma$
of $G$ such that $(u_2,v_2)=(\sigma(u_1),\sigma(v_1))$.
\begin{Lemma}\label{lemma2}
 Let $u$ and $v$ be two distinct vertices in a distance-transitive graph $G$. Then $R(u,v)=F(u,v)$.
\end{Lemma}
\proof Note that all distance-transitive graphs are
vertex-transitive. Then $G$ is non-rigid. Take any $x\in F(u,v)$.
Then $O_x(u)\ne O_x(v)$. If $d(x,u)=d(x,v)$, then there is an
automorphism $\sigma$ of $G$ such that
$(x,v)=(\sigma(x),\sigma(u))$, which implies that $v\in O_x(u)$, and
so $O_x(u)=O_x(v)$, a contradiction. Hence, we have $F(u,v)\subseteq
R(u,v)$. Therefore, we get the desired result by Lemma~\ref{lemma}.
\qed

\medskip

According to the above lemma, we get the following result
immediately.

\begin{Theorem}\label{distancetran}
  For a distance-transitive graph $G$, we have $fix_f(G)=dim_f(G)$.
\end{Theorem}
The {\em Hamming graph}, denoted by $H_{n,k}$, has the vertex set
$\{(x_1,\ldots,x_n)|1\leq x_i\leq k, 1\leq i\leq n\}$, with two
vertices being adjacent if they differ in exactly one co-ordinate.
Let $X$ be a set of size $n$, and let ${X\choose k}$ denote the set
of all  $k$-subsets of $X$.
 The {\em Johnson graph}, denoted
by $J(n,k)$, has ${X\choose k}$ as the vertex set, where two
$k$-subsets are adjacent if their intersection has size $k-1$.\\

It is well-known that $H_{n,k}$ and $J(n,k)$ are
distance-transitive. The fractional metric dimension of $H_{n,2}$
was computed in \cite{AM}. Feng at el. \cite{FLW} compute
$dim_f(H_{n,k})$ for $k\ge 3$ and $dim_f(J(n,k))$. Combining all
these results and Theorem~\ref{distancetran}, we get

\begin{Corollary}
Let $n$ and $k$ be positive integers at least $2$.

  {\rm(i)} $fix_f(H_{n,k})=\left\{\begin{array}{ll}
                                       2, & \textup{if }k=2, \\
                                       \frac{k}{2}, & \textup{if }k\ge 3.
                                     \end{array}\right.$

{\rm(ii)} If $n\ge 2k$, then
\begin{eqnarray*}
fix_f(J(n,k))=\left\{
\begin{array}{ll}
3,& \textup{ if }(n,k)=(4,2),\\
\frac{35}{17},& \textup{ if }(n,k)=(8,4),\\
\frac{n^2-n}{2kn-2k^2},& \textup{ otherwise. }
\end{array}\right.
\end{eqnarray*}
\end{Corollary}
The next result is a generalization of Theorem \ref{mthm}. Note that
for $C(G)=\emptyset$, the next theorem coincides with Theorem
\ref{mthm}.
\begin{Theorem}\label{t2}
Let $G$ be a connected graph of order $n$. Then
$fix_f(G)=\frac{n-|C(G)|}{2}$ if and only if each vertex in $A(G)$
has a twin.
\end{Theorem}
\begin{proof}
By Lemma \ref{cg}, $fix_f(G)\leq \frac{n-|C(G)|}{2}$. Let $h$ be any
fixing function of $G$. Then by Lemma \ref{twins}, $h(u)+h(v)\geq 1$
for all $u,v\in A(G)$. Adding these $n-|C(G)|$ inequalities, we get
$fix_f(G)\geq |h|\geq \sum\limits_{i=1}^{n-|C(G)|}h(u_i)\geq
\frac{n-|C(G)|}{2}$. Therefore, $fix_f(G)=\frac{n-|C(G)|}{2}$.
Converse part of this theorem is straight forward from Theorem
\ref{mthm}.
\end{proof}
%\begin{Corollary} For $n\ge 2$, we have
%$fix_f(K_{1,n})=\frac{n}{2}$.
%\end{Corollary}
The \emph{friendship graph} $F_n$ can be constructed by joining $n$
copies/blocks of the cycle graph $C_3$ with a common vertex.
\begin{Corollary}
For friendship graph $F_n$, $fix_f(F_n)=n$.
\end{Corollary}
\begin{proof}
For $1\leq i\leq n$, let $H_i=(x,a_i,b_i,x)$ be the $n$ blocks of
$F_n$. Note that $C(F_n)=\{x\}$ and $F(a_i,b_i)=\{a_i,b_i\}$. By
Lemma \ref{twins}, $a_i$ and $b_i$ are twins. Hence by Theorem
\ref{t2}, $fix_f(F_n)=\frac{|V|-|C(F_n)|}{2}=n$.
\end{proof}
The \emph{fan graph} $F_{1,n}$ of order $n+1$ is defined as the join
graph $K_1+P_n$.
\begin{Corollary} For fan graph $F_{1,n}$ with $n\geq 3$,
$$fix_f(F_{1,n}) =\left\{
  \begin{array}{ll}
    2,   \,\,\,\     & \mbox{if}\,\,\,n=3, \\
    1,   \,\,\,\     & \mbox{if}\,\,\,n\geq4.
  \end{array}
\right. $$
\end{Corollary}
\begin{proof}
Note that each vertex in $F_{1,3}$ has a twin, so by Theorem
\ref{mthm}, $fix_f(F_{1,3})=2$. Now, suppose $n\geq4$. Take a vertex
$u\in V(F_{1,n})$ of degree $2$. Then for each $(x,y)\in
V_a(F_{1,n})$, we have $u\in F(x,y)$. Since $F_{1,n}$ is not a rigid
graph so one has $fix(F_{1,n})=1$, which implies that
$fix_f(F_{1,n})=1$, as desired.
\end{proof}
%\begin{Theorem}
%Let $T_r=(V,E)$ be a tree with $r\geq 2$, then
%$fix_f(T_r)\leq\frac{r}{2}$.
%\end{Theorem}
%\begin{proof}
%Note that $|F(x,y)|\geq 2$ for all $(x,y)\in
%V_a(T_r)$. The function $g:V\rightarrow V$ defined by
%$$g(u) =\left\{
%  \begin{array}{ll}
%    \frac{1}{2},   \,\,\,\     & \mbox{if}\,\,\,u\,\, \mbox{is a non-rigid leaf},\\
%     0,   \,\,\,\     & \mbox{otherwise}.
%  \end{array}
%\right.$$ is a fixing function of $T_r$ with $|g|=\frac{r}{2}$.
%Hence, $fix_f(T_r)\leq \frac{r}{2}$.
%\end{proof}

For $v\in V(G)$, $G-v$ is known as the vertex deleted subgraph of
$G$ obtained by deleting $v$ from the vertex set of $G$ along with
its incident edges.
\begin{Proposition}\label{propcut}
For a connected graph $G$, $fix_f(G)-1\leq fix_f(G-v)$, where $v$ is
a vertex of $G$.
\end{Proposition}
\begin{proof}
Let $g:V(G-v)\rightarrow V(G-v)$ be a fixing function of $G-v$ such
that $fix_f(G-v)=|g|$. Now, the function $g^{\prime}:V(G)\rightarrow
V(G)$ defined by
$$g^{\prime}(u) =\left\{
  \begin{array}{ll}
    g(u),   \,\,\,\     & \mbox{if}\,\,\,u\neq v,\\
    1,   \,\,\,\     & \mbox{if}\,\,\,u=v.
  \end{array}
\right.$$ is a fixing function of $G$ and hence $fix_f(G)\leq
|g^{\prime}|$. Thus $fix_f(G-v)=|g|=|g^{\prime}|-1\geq fix_f(G)-1$.
\end{proof}
%For a tree $T$, a vertex of degree one is called a leaf and a leaf
%$l$ is said to be rigid if there exists no automorphism $\alpha$ of
%$T$ such that $\alpha(l)=x$, for any $x \in V(T)\setminus\{l\}$.
In the following result, the fractional fixing number of trees has
been computed. Let $T=(V(T),E(T))$ be an $n$-vertex non-path tree
with $n\geq 4$, then $fix_f(T)\geq 1$. A vertex of degree one is
called a leaf.
\begin{Theorem}\label{tre}
The fractional fixing number of a tree $T$ with $n$ vertices
satisfies the following statements:
\begin{enumerate}
  \item $0\leq fix_f(T)\leq \frac{n-1}{2}$ and both bounds are
  tight.
  \item Given $n,k\in \mathbb{N}$ with $2\leq k\leq n-1$ and $k\neq
  n-2$, there exists a tree $T$ of order $n$ such that
  $fix_f(T)=\frac{k}{2}$.
  \item A tree $T$ such that $fix_f(T)=0$ can only exists if $n=1$
  or $n\geq 7$.
\end{enumerate}
\end{Theorem}
\begin{proof}
\begin{enumerate}
  \item By definition, $fix_f(T)\geq 0$. Furthermore, $T$ contains
  at most $n-1$ leaves which implies, by Theorem \ref{t2}, that
  fractional fixing number is at most $\frac{n-1}{2}$. Upper bound is
  sharp for star graph, $K_{1,n-1}$.
  \item Consider $n,k\in \mathbb{N}$ with $2\leq k\leq n-3$, a $u-v$
  path $P_{n-k}$ with  group of leaves $\{v_1,v_2,...,v_k\}$ hanging
  from $v$. Theorem \ref{t2} implies its fractional fixing number is
  $\frac{k}{2}$. The star $K_{1,n-1}$ serves as example for $k=n-1$.
  \item It was proved in \cite{cac} that a tree $T$ such that $fix(T)=0$
  can only exists if $n=1$ or $n\geq 7$. This implies $fix_f(T)=0$ can exists only if $n=1$ or $n\geq 7$.
  \end{enumerate}
\end{proof}
There exists families of graphs for which for which $dim_f(T)$ and
$fix_f(T)$ are equal. Consider a tree $T$ formed by connecting a
single vertex $u$ to $k$ paths denoted by
$P_m,P_{m+1},...,P_{m+k-1}$ with lengths $m,m+1,...,m+k-1$,
respectively. It is clear that such a tree is a rigid graph and
$fix_f(T)=0$ and it was shown in \cite{KS} that
$dim_f(T)=\frac{k}{2}$. Hence there exist graphs for which the
difference between $dim_f(T)$ and $fix_f(T)$ can be arbitrarily
large.

\begin{Corollary}
For the wheel $W_n$, $n\geq 5$, we have
$$fix_f(W_n)=\left\{\begin{array}{ll}
                                       \frac{n}{n-2}, & \textup{if $n-1$ is even}, \\
                                       \frac{n}{n-3}, & \textup{if $n-1$ is odd}.
                                     \end{array}\right.
$$
\end{Corollary}
\begin{proof}
Let $V(W_n)=\{u_1,u_2,...,u_{n-1},u\}$ where $u$ is the center of
the wheel and $C_{n-1}=(u_1,u_2,...,u_{n-1})$, $n\geq 4$, is the
rim. Since the center vertex $u$ does not belong to $A(W_n)$,
therefore the automorphism group of $W_n$ is same as that of cycle
$C_{n-1}$. Therefore, we have the required result.
\end{proof}

In \cite{AM}, it was shown that $dim_f(W_n)=\frac{n-1}{4}$ for
$n\geq 7$. Please note $dim_f(W_n)-fix_f(W_n)\rightarrow \infty$ as
$n\rightarrow\infty$.
\section{Fractional fixing number of corona product of graphs}
Let $G$ and $H$ be two graphs with $|V(G)|=m$ and $|V(H)|=n$.
\emph{Corona product} of $G$ and $H$, denoted by $G\odot H$, is the
graph obtained from $G$ and $H$ by taking one copy of $G$ and $m$
copies of $H$ and joining each vertex from the $i^{th}$-copy of $H$
by an edge with the $i^{th}$-vertex of $G$. Let $u\in V(G)$ then
$H_u$ be the copy of $H$ corresponding to the $u$-vertex of $G$. We
write $H_u = \{(u,v):v\in V(H)\}$ for $u\in V(G)$. For $x,y\in
V(G\odot H)$, the fixing neighborhood of $x,y$ is denoted by
$F_{G\odot H}(x,y)$ and $F_G(x,y)$ denotes the fixing neighborhood
of $x,y$ in $G$.
\begin{Lemma}\label{1c}
Let $G$ be a connected graph of order $m\geq 2$ and $H$ be an
arbitrary graph. Let $(x,y)\in V_a(G\odot H)$.
\begin{enumerate}
  \item If $\{(x,y)\}\subseteq V_a(H_u)$ for some $u\in V(G)$, say $x=(u,v_1)$ and
$y=(u,v_2)$, then
$$F_{G\odot H}(x,y)=F_H(v_1,v_2).$$
  \item If $\{x,y\}\not\subseteq V_a(H_u)$ for any $u\in V(G)$, then there
exists a vertex $u_0$ of $G$ such that $H_{u_0}\subseteq F_{G\odot
H}(x,y).$
\end{enumerate}
\end{Lemma}

\begin{proof}
(1) If $\{(x,y)\}\subseteq V_a(H_u)$ for any $u\in V(G)$, then it is
clear that there exists $\alpha\in \Gamma_{V(G\odot H)\setminus
H_u}$ such that $\alpha(x)=y$, therefore we have $F_{G\odot
H}(x,y)\subseteq H_u$. Note that $(r,s)\in F_{G\odot H}(x,y)$ is
equivalent to $s\in F_H(v_1,v_2)$. Hence, the desired result follows.\\
(2) Note that for $x\in V(G)$ and $y\in H_u$ or $y\in V(G)$ and
$x\in H_u$, $(x,y)\notin V_a(G\odot H)$. Now,
we have two cases:\\
Case 1: Let $(x,y)\in V_a(G)$, then for $\alpha\in \Gamma(G)$,
$\alpha(x)=y$ if and only if $\alpha(H_x)=H_y$. This implies that
$H_x$ and
$H_y\subseteq F_{G\odot H}(x,y)$.\\
Case 2: Let $x\in H_{u_1}$ and $y\in H_{u_2}$ for two distinct
vertices $u_1,u_2\in V(G)$. It is clear that by fixing any vertex of
$H_{u_1}$ or $H_{u_2}$, $x$ cannot be mapped on $y$. Therefore,
$H_{u_1}\subseteq F_{G\odot H}(x,y)$.
\end{proof}
Fixing number of corona product of non-rigid graphs has been studied
by Javaid {\it et al.} and they proved that $fix(G\odot H)=mfix(H)$
\cite{iijav}. It is interesting to note that a similar result is
true for the fractional fixing number of $G\odot H$ as well.
\begin{Theorem}\label{2c}
Let $G$ be a connected graph and $H$ be a non-rigid graph with
$|V(G)|=m\geq 2$. Then $fix_f(G\odot H)= mfix_f(H)$.
\end{Theorem}
\begin{proof}
Let $g$ be a fixing function of $G\odot H$ with $|g|=fix_f(G\odot
H)$. For each $u\in V(G)$, define $g_u:V(H)\rightarrow [0,1]$ such
that $ v\mapsto g((u,v)).$ For $(v_1,v_2)\in V_a(H)$, by Lemma
\ref{1c},
$$g_u(F_H(v_1,v_2))=\sum_{v\in F_H(v_1,v_2)} g((u,v))=g(F_{G\odot H}((u,v_1),(u,v_2)))\geq1,$$
which implies that
$$|g_u|\geq fix_f(H).$$
Since $V(H)\subseteq V(G\odot H)$, we have
$$|g|\geq \sum_{u\in V(G)}|g_u|.$$
Hence, $fix_f(G\odot H)\geq mfix_f(H).$ Now, we show that
$fix_f(G\odot H)\leq mfix_f(H).$ Note that for any pair $(u,v) \in
V_a(G)$, $H_u\subseteq F_{G\odot H}(u,v)$ and $H_v\subseteq
F_{G\odot H}(u,v)$. Let $h:V(H)\rightarrow[0,1]$ is a fixing
function of $H$ such that $|h|=fix_f(H)$. Define
$$
h': V(G\odot H)\rightarrow[0,1],\qquad w\mapsto\left\{
\begin{array}{ll}
h(y),& \textup{if } w=(x,y),\\
0,& \textup{if } w\in V(G).
\end{array}
\right.
$$ Note that $h'$ is a fixing function of $G\odot H$. Hence, $fix_f(G\odot H)\leq mfix_f(H)$ and the
result follows.
\end{proof}
\begin{Theorem}
  Let $G$ be a connected graph of order at least $2$ and $H$ be a rigid graph. Then
  $$
  fix_f(G\odot H)=fix_f(G).
  $$
\end{Theorem}
\proof
If $G$ is a rigid graph, then $G\odot H$ is a rigid graph,
and so $fix_f(G\odot H)=0=fix_f(G)$. In the following, suppose that
$G$ is not a rigid graph. Then $G\odot H$ is not a rigid graph.

Let $g$ be a fixing function of $G\odot H$ with $|g|=fix_f(G\odot
H)$. Define
$$
g': V(G)\rightarrow[0,1], \qquad u\mapsto g(u)+\sum_{v\in
H_u}g((u,v)).
$$
For any $(u_1,u_2)\in V_a(G)$, we have
$$
F_{G\odot H}(u_1,u_2)=F_G(u_1,u_2)\cup\bigcup_{u\in
F_{G}(u_1,u_2)}H_u,
$$
which implies that $g'(F_G(u_1,u_2))=g(F_{G\odot H}(u_1,u_2))\ge 1$,
and so $g'$ is a fixing function of $G$ with $|g'|=|g|$. Therefore,
one has $fix_f(G)\le fix_f(G\odot H)$.

Let $h$ be a fixing function of $G$ with $|h|=fix_f(G)$. Define
$$
h': V(G\odot H)\rightarrow[0,1],\qquad u\mapsto\left\{
\begin{array}{ll}
h(u),& \textup{if } u\in V(G),\\
0,& \textup{otherwise.}
\end{array}
\right.
$$
Note that $H$ is a rigid graph. For $(x,y)\in V_a(G\odot H)$, we
have $(x,y)\in V_a(G)$, or $(x,y)=((u_1,v),(u_2,v))$ for some
$((u_1,u_2),v)\in V_a(G)\times V(H)$. It follows that $h'$ is a
fixing function of $G\odot H$ with $|h'|=|h|$. Thus, we get
$fix_f(G\odot H)\le fix_f(G)$. Hence, the desired result follows.
\qed

\medskip

Let $H$ be a graph with maximum degree less than $|V(H)|-1$. If $H$
is a rigid graph, then $K_1\odot H$ is a rigid graph, and so
$fix_f(K_1\odot H)=0=fix_f(H)$. If $H$ is not a rigid graph, then
$fix_f(K_1\odot H)=fix_f(H)$ by a similar proof of Theorem \ref{2c}.
Consequently, we have

\begin{Theorem}
  Let $H$ be a graph with maximum degree less than $|V(H)|-1$. Then
  $$
  fix_f(K_1\odot H)=fix_f(H).
  $$
\end{Theorem}

Now, we compute $fix_f(K_1\odot H)$ if $H$ has maximum degree
$|V(H)|-1$.

\begin{Theorem}
Let $H$ be a graph with maximum degree $|V(H)|-1$. Suppose the
number of vertices with degree $|V(H)|-1$ in $H$ is $k$. Then
  $$
  fix_f(K_1\odot H)=\left\{\begin{array}{ll}
                             fix_f(H)+1, & \textup{if } k=1, \\
                             fix_f(H)+\frac{1}{2} & \textup{if } k\ge 2.
                           \end{array}\right.
  $$
\end{Theorem}
\proof Note that there are $k+1$ vertices with degree $|V(H)|$ in
$K_1\odot H$ and all of them are twins. The rest of the proof is
similar to the proof of Theorem \ref{2c}.\qed
\section{Fractional Fixing Number of Composition Product of Graphs}
Let $G$ and $H$ be two graphs. The \emph{composition product} of $G$
and $H$, denoted by $G[H]$, is the graph with vertex set $V(G)\times
V(H)=\{(u,v): u\in V(G)$ and $v\in V(H)\}$, where $(u,v)$ is
adjacent to $(x,y)$ whenever $ux\in E(G)$ or $u=x$ and $vy\in E(H)$.
For any vertex $u\in V(G)$ and $v\in V(H)$, we define the vertex set
$H(u)=\{(u,y)\in V(G[H]):y\in V(H)\}$ and $G(v)$ = $\{(x,v)\in
V(G[H]):x\in V(G)\}$.

\par Let $G$ be a connected graph and $H$ be an arbitrary
graph containing $k\geq 1$ components $H_{1},H_{2},\cdots,H_{k}$
with $|V(H_{j})|\geq 2$ for each $j=1,2,\ldots,k$. For any vertex
$u\in V(G)$ and $1\leq i\leq k$, we define the vertex set $H_{i}(u)
=\{(u,v)\in V(G[H]):v\in V(H_{i})\}$. Let $|V(H_{i})|=m_{i}$, $1\leq
i\leq k$. From the definition of $G[H]$, it is clear that for every
$(u, v)\in V(G[H])$, $deg_{G[H]}(u, v)= deg_{G}(u)\cdot |V(H)|+
deg_{H}(v)$. If $G$ is a disconnected graph having $k\geq 2$
components $G_1$, $G_2$, \ldots, $G_k$, then $G[H]$ is also a
disconnected graph having $k$ components such that $G[H]=G_1[H]\cup
G_2[H]\cup \ldots \cup G_k[H]$ and each component $G_i[H]$ is the
composition product of connected component $G_i$ of $G$ with $H$,
therefore throughout this section, we will assume $G$ to be
connected. For $x,y\in V(G[H])$, the fixing neighborhood of $x,y$ is
denoted by $F_{G[H]}(x,y)$. For a subgraph $Q$ of a graph $G$,
$\mathcal{F}(Q)=\{x\in V(G):x\in F(u,v) \,\,\mbox{for} (u,v)\in
V_a(Q)\}$.
\begin{Lemma}\cite{iijav}\label{lma s1}
Let $G$ and $H$ be two non-rigid graphs. For two distinct vertices
$u,v \in V(G)$, if $z\in H(v)$, then $z\not\in F_{G[H]}(x,y)$ for
$(x, y) \in V_a(H(u))$.
\end{Lemma}
\begin{Lemma}\cite{iijav}\label{lma s2}
Let $G[H]$ be the composition product of two non-rigid graphs $G$
and $H$. Let $H_1, H_2,\ldots, H_k$, $k\geq1$, be the non-trivial
components of $H$. Then for $u\in V(G)$ and $x\in H_j(u)$, $x\not\in
\mathcal{F}(H_i(u))$, $1\leq i, j\leq k$ and $i\neq j$.
\end{Lemma}
\begin{Lemma}\label{lexico 1} Let $G$ and $H$ be two non-rigid graphs. Let
$H_{1},H_{2},\cdots,H_{k}$, $k\geq1$, be the non-trivial components
of $H$. Let $\{x,y\}\subseteq V_a(G[H])$.
\\(1)If $\{x,y\}\subseteq V_a(H_i(u))$, for some $u\in V(G)$, say $x=(u,
v_1)$ and $y=(u,v_2)$, then $F_{G[H]}(x,y)= \{u\}\times
F_{H_i}(v_1,v_2)$.
\\(2) If $\{x,y\}\subseteq V_a(G(v))$ for some $v\in V(H)$ say
$x=\{u_1,v\}$ and $y=\{u_2,v\}$ then $F_{G[H]}\{x,y\}\supseteq
F_G(u_1,u_2)\times \{v\}$.
\end{Lemma}
\begin{proof}$(1)$ holds from Lemma \ref{lma s1} and \ref{lma s2}
and (2) is obvious.
\end{proof}
\begin{Lemma}\label{lexico 2}Let $G$ and $H$ be two non-rigid graphs. Let
$H_{1},H_{2},\cdots,H_{k}$, $k\geq1$, be the non-trivial components
of $H$. Let $g$ be a fixing function of $G[H]$ such that
$fix_f(G[H])=|g|$. Then for $u\in V(G)$, $g_i^u:V(H_i(u))\rightarrow
[0,1]$, $(u,x)\mapsto g(u,x)$, is a fixing function of $H_i(u)$.
Moreover, if $h_i$ is a fixing function of $H_i$ such that
$fix_f(H_i)=|h_i|$ then $|g_i^u|\geq |h_i|$.
\end{Lemma}
\begin{proof} For $(u,v_1),(u,v_2)\in V_a(H_i(u))$, by Lemma
\ref{lexico 1},

$g_i^u(F_{H_i(u)}(u,v_1),(u,v_2))=\sum\limits_{w\in
F_{H_i}(v_1,v_2)}g(u,w)=g(F_{G[H]}(u,v_1),(u,v_2))\geq1$, \\which
implies that $g_i^u$ is a fixing function of $H_i(u)$.

Suppose $|g_i^u|<|h_i|$. Then there exist two distinct vertices
$x,y$ in $H_i(u)$ such that $g_i^u(F_{H_i(u)}(x,y))<1$. Then
$g_i^u(F_{H_i(u)}(x,y))=g(F_{G[H]}(x,y))<1$, which contradicts the
fact that $g$ is a fixing function of $G[H]$. Hence, $|g_i^u|\geq
|h_i|$.
\end{proof}
\begin{Theorem}Let $G$ and $H$ be two non-rigid graphs of orders $m$ and $n$ respectively. Let
$H_{1},H_{2},\cdots,H_{k}$, $k\geq1$, be the non-trivial components
of $H$. Then
$$m\left(\sum\limits_{i=1}^kfix_f(H_i)\right)\leq fix_f(G[H])\leq \frac{mn}{2}.$$
\end{Theorem}
\begin{proof} Upper bound follows from Corollary \ref{cro1n}. Let
$g$ be a fixing function of $G[H]$ with $|g|=fix_f(G[H])$. Then for
$u\in V(G)$, $g_i^u:V(H_i(u))\rightarrow [0,1]$, $(u,x)\mapsto
g(u,x)$, is a fixing function of $H_i(u)$ and $|g_i^u|\geq
fix_f(H_i)$, by Lemma \ref{lexico 2}. Then $|g|\geq
\sum\limits_{u\in V(G)}\sum\limits_{i=1}^k|g_i^u|$. Hence,
$fix_f(G[H])\geq m\left(\sum\limits_{i=1}^kfix_f(H_i)\right)$.
\end{proof}
\begin{Theorem}Let $G$ be a non-rigid and $H$ a rigid graph. Then
$fix_f(G[H])=fix_f(G)$.
\end{Theorem}
\begin{proof} For $(u_1,u_2)\in V_a(G)$, $F_{G[H]}((u_1,x),(u_2,x))=F_G(u_1,u_2)\times
V(H)$ for any $x\in V(H)$. Let $g$ be a fixing function of $G[H]$
with $|g|=fix_f(G[H])$. Define $g':V(G)\rightarrow[0,1]$, $u\mapsto
\sum\limits_{v\in V(H)}g(u,v)$. Then

$g'(F_G(u_1,u_2))=\sum\limits_{v\in
V(H)}g(F_{G[H]}((u_1,v),(u_2,v)))\geq 1$, which implies that $g'$ is
a fixing function of $G$. Hence, $fix_f(G[H])\geq fix_f(G)$.

Let $h$ be a fixing function of $G$ with $|h|=fix_f(G)$. Define $$
h': V(G[H])\rightarrow[0,1],\qquad w\mapsto\left\{
\begin{array}{ll}
h(u),& \textup{if } w=(u,v)\,\, \mbox{for some fixed}\,\, v\in V(H),\\
0,& \textup{otherwise }.
\end{array}
\right.
$$
For $(x,y)\in V_a(G[H])$, $x=(b,z)$ and $y=(c,z)$ for some $(b,c)\in
V_a(G)$,

$h'(F_{G[H]}(x,y))=h'(F_{G[H]}((b,z),(c,z))=h(F_G(b,c))\geq1$, which
implies that $h'$ is a fixing function of $G[H]$. Hence,
$fix_f(G)\geq fix_f(G[H])$. Hence, the desired result follows.
\end{proof}
\section{Summary and Conclusion}
In this paper, the concept of the fractional fixing number of graphs
has been studied. We have also introduced integer programming
formulation of the fractional fixing number of graphs. Graphs with
fractional fixing number $fix_f(G)=\frac{|V(G)|}{2}$ have also been
characterized. The fractional fixing number of some families of
graphs, corona product and composition product of graphs have also
obtained. However, it remains to determine the fractional fixing
number of several other families of graphs and graph products.
Metric dimension and fixing number are closely related parameters.
There are several graphs for which the study of the fractional
fixing number of graphs is similar to that of the fractional metric
dimension of graphs.

\end{document}